\documentclass[11pt]{article}

\usepackage[pdftex]{graphicx}

\usepackage{amssymb}
\usepackage{amsthm}

\usepackage{accents}

\makeatletter
\renewcommand\section{\@startsection{section}{1}{\z@}%
                                  {-3.5ex \@plus -1ex \@minus -.2ex}%
                                  {2.3ex \@plus.2ex}%
                                  {\normalfont\large\bfseries}}
\makeatother

\DeclareGraphicsExtensions{.jpg}
\begin{document}

\title{Efficient domination in regular graphs}

\author{Misa Nakanishi \thanks{E-mail address : nakanishi@2004.jukuin.keio.ac.jp}}
\date{}
\maketitle

\newtheorem{thm}{Theorem}[section]
\newtheorem{lem}{Lemma}[section]
\newtheorem{prop}{Proposition}[section]
\newtheorem{cor}{Corollary}[section]
\newtheorem{rem}{Remark}[section]
\newtheorem{conj}{Conjecture}[section]
\newtheorem{claim}{Claim}[section]
\newtheorem{fact}{Fact}[section]
\newtheorem{obs}{Observation}[section]

\newtheorem{defn}{Definition}[section]
\newtheorem{propa}{Proposition}
\renewcommand{\thepropa}{\Alph{propa}}
\newtheorem{conja}[propa]{Conjecture}

\section{Notation}

In this paper, a graph $G$ is finite, undirected, and simple with the vertex set $V$ and edge set $E$. We follow \cite{Diestel} for basic notation. For a vertex $v \in V(G)$, the open neighborhood, denoted by $N_G(v)$, is $\{u \in V(G)\colon\ uv \in E(G)\}$, and the closed neighborhood,  denoted by $N_G[v]$, is $N_G(v) \cup \{v\}$, also for a set $W \subseteq V(G)$, let $N_G(W) := \bigcup_{v \in W} N_G(v)$ and $N_G[W] := N_G(W) \cup W$. For a vertex $v \in V(G)$, let $N_G^{2}(v) := \{ w \in V(G)\colon\ d_G(v, w) = 2 \}$. A {\it dominating set} $X \subseteq V(G)$ is such that $N_G[X] = V(G)$. A dominating set $X$ is {\it efficient} if $X$ is independent and every vertex of $V(G) \setminus X$ is adjacent to exactly one vertex of $X$.

\section{Efficient domination in regular graphs}

\begin{thm}\label{thm1}
Let $G$ be a connected regular graph. Whether $G$ has efficient dominating sets is determined in polynomial time.
\end{thm}

\begin{proof}
Let $i$ be an integer such that $i \geq 0$. Let $A_i \subseteq V(G)$. Let $X$ be an efficient dominating set of $G$ such that $X \subseteq A_i$ if there exists.
Let $\mathcal{X}_i$ be the set of all $X$.
The next facts are easily verified from the definition of $X$.

\begin{fact}\label{fact}
Let $v \in A_i$.
If $(N_G(c) \setminus N_G(v)) \cap A_i = \emptyset$ for some $c \in N_G^{2}(v) \ne \emptyset$, then $v \not \in X$ for all $X \in \mathcal{X}_i$. (We say that the vertex $v$ satisfies Fact \ref{fact}.)
\end{fact}

\begin{proof}
Let $v \in X$ for some $X \in \mathcal{X}_i$. Suppose $(N_G(c) \setminus N_G(v)) \cap A_i = \emptyset$ for some $c \in N_G^{2}(v) \ne \emptyset$. Since $X$ is a dominating set and $X \in \mathcal{X}_i$, $(N_G[c] \setminus (N_G(c) \setminus N_G(v)) ) \cap X \ne \emptyset$. It contradicts that $X$ is efficient. We obtain the statement.
\end{proof}

We say that $G$ is covered by paths {\it sequentially} if $G$ is covered by not necessarily disjoint paths but paths with overlap.

\begin{fact}\label{fact3-1}
If $Y \subseteq V(G)$ is a dominating set of $G$, then $G$ is covered by paths that have vertices of $Y$ at every three vertices sequentially; these paths are represented by $y_{1}b_{1,1}b_{1,2}y_{2}b_{2,1}b_{2,2} \cdots y_{h}$ or $y_{1}b_{1,1}b_{1,2}y_{2}b_{2,1}b_{2,2} \cdots y_{h}b_{h,1}$ such that $y_1, \cdots, y_h \in Y$ and $b_{1,1}, b_{1,2}, \cdots, b_{h,1}, b_{h,2} \in V(G) \setminus Y$ for some $h$ such that $h \geq 1$.
\end{fact}

\begin{fact}\label{fact3}
If $G$ is covered by paths that have labels at every three vertices sequentially; these paths are represented by $y_{1}b_{1,1}b_{1,2}y_{2}b_{2,1}b_{2,2} \cdots y_{h}$ or $y_{1}b_{1,1}b_{1,2}y_{2}b_{2,1}b_{2,2} \cdots y_{h}b_{h,1}$ such that $y_1, \cdots, y_h$ are labeled and $b_{1,1}, b_{1,2}, \cdots, b_{h,1}, b_{h,2}$ are not labeled for some $h$ such that $h \geq 1$, then the set of all labeled vertices is a dominating set of $G$. 
\end{fact}

\begin{prop}\label{fact4}
Suppose $\mathcal{X}_i \ne \emptyset$. For $a \in A_i$, let $A_{i + 1} := A_i \setminus (N_G(a) \cup N_G^{2}(a))$. 
Suppose that the following steps are applied to $A_{i + 1}$. 
\begin{enumerate}
\item[(a-1)] Drop a vertex $v \in A_{i + 1}$ that satisfies Fact \ref{fact} from $A_{i + 1}$, and let $A_{i + 2} := A_{i + 1} \setminus \{v\}$.
\item[(a-2)] Increment $i$, and go to (a-1).
\end{enumerate}
For the final set $A_j$ ($j \geq i + 1$), let $B_i(a) := A_j$.
Then $a \not \in X$ for all $X \in \mathcal{X}_i$ if and only if $B_i(a) = \emptyset$.
\end{prop}

\begin{proof}
Let $a \in X$ for some $X \in \mathcal{X}_i$. Now, $X \subseteq A_{i + 1}$. By Fact \ref{fact}, if $B_i(a) = \emptyset$ then $a \not \in X$ for all $X \in \mathcal{X}_i$. Conversely, let $a \not \in X$ for all $X \in \mathcal{X}_i$ and $B_i(a) \ne \emptyset$. By Fact \ref{fact}, $B_i(a)$ is a dominating set of $G$. Set $A_i = X_1 \cup \cdots \cup X_l \cup F$ where $\mathcal{X}_i = \{X_1, \cdots, X_l\}$ ($l \geq 1$). If for some $X \in \mathcal{X}_i$, $X \subseteq B_i(a)$, then $a \in X$. That is, for all $h \in \{1, 2, \cdots, l\}$, $X_h \not \subseteq B_i(a)$. 

Let $y_1 := a$. Let $x_1 \in N_G(y_1) \cap X_1$. It suffices that $|X_1| \geq 2$. By Fact \ref{fact}, if $c \in N_G^{2}(y_1)$ then $(N_G(c) \setminus N_G(y_1)) \cap B_i(a) \ne \emptyset$. Let $y_2 \in (N_G(c) \setminus N_G(y_1)) \cap B_i(a)$ for some $c \in N_G^{2}(y_1)$. By the same argument as Fact \ref{fact}, we define recursively $y_k \in B_i(a)$ for $2 \leq k \leq p$ $(p \geq 2)$. Let $Y := \{ y_k\colon\ 1 \leq k \leq p\}$. Here, $y_{i_1}, \cdots, y_{i_h} \in Y$ for some $i_1, \cdots, i_h$ such that $1 \leq i_1, \cdots, i_h \leq p, h \geq 1$ are located at every three vertices on a path in $G$, represented by $y_{i_1}b_{1,1}b_{1,2}y_{i_2}b_{2,1}b_{2,2} \cdots y_{i_h}$ or $y_{i_1}b_{1,1}b_{1,2}y_{i_2}b_{2,1}b_{2,2} \cdots y_{i_h}b_{h,1}$ such that $b_{1,1}, b_{1,2}, \cdots, b_{h,1}, b_{h_2} \in V(G) \setminus Y$, and $G$ is covered by such paths sequentially. (It is possible because $B_i(a)$ is a dominating set and Fact \ref{fact3-1}.) Set $X_1 = \{x_1, x_2, \cdots, x_q\}$ $(q > 1)$. More strictly, we prove that $y_k \in N_G[x_k] \cap Y$ for $1 \leq k \leq q$ and $p = q$ are also possible.

\begin{claim}
It is possible to take $Y$ as $y_k \in N_G[x_k] \cap Y$ for $1 \leq k \leq q$ and $p = q$.
\end{claim}

\begin{proof}
Suppose $y_{j_1}, y_{j_2} \in N_G[x_k] \cap Y$ for some $j_1$ and $j_2$ such that $1 < j_1 < j_2 \leq p$ and for some $k$ such that $1 < k \leq q$. It suffices that $y_{j_2} \in N_G(x_k)$. Let $P_1$ be a path such that vertices of $Y$ and vertices of $X_1$ are adjacent and located at every three vertices on the sequence, and $y_1, y_{j_2} \in V(P_1)$. Set $y_{j_2}P_{1}y_{1} = y_{j_2}a_{w}x_{w}y_{w} \cdots a_{1}x_{1}y_{1}$ for some $w$ such that $w \geq 1$, $y_1, \cdots, y_w \in Y$, $x_1, \cdots, x_w \in X_1 \setminus Y$, and $a_1, \cdots, a_w \in V(G) \setminus (X_1 \cup Y)$. Case (I) $y_{j_1} \in N_G(x_k)$. Case (I-i) $d_G(y_{j_1}, y_{j_2}) = 2$. By Fact \ref{fact}, it follows that $(N_G(y_{j_2}) \setminus N_G(y_{j_1})) \cap B_i(a) \ne \emptyset$. Let $y_{\alpha} \in (N_G(y_{j_2}) \setminus N_G(y_{j_1})) \cap B_i(a)$. Case (I-i-1) $d_G(a_w, y_{\alpha}) = 2$. By Fact \ref{fact}, it follows that $(N_G(a_w) \setminus N_G(y_{\alpha})) \cap B_i(a) \ne \emptyset$. Now, we can take $y_{\beta} \in (N_G(a_w) \setminus N_G(y_{\alpha})) \cap B_i(a)$ and avoid $y_{j_2}$; let $Y := (Y \cup \{y_{\beta}\}) \setminus \{y_{j_2}\}$. Case (I-i-2) $d_G(a_w, y_{\alpha}) = 1$. Now, we can take $y_{\alpha}$ and avoid $y_{j_2}$; let $Y := (Y \cup \{y_{\alpha}\}) \setminus \{y_{j_2}\}$. Case (I-ii) $d_G(y_{j_1}, y_{j_2}) = 1$ and $d_G(a_w, y_{j_1}) = 2$. By Fact \ref{fact}, it follows that $(N_G(a_w) \setminus N_G(y_{j_1})) \cap B_i(a) \ne \emptyset$. Now, we can take $y_{\alpha} \in (N_G(a_w) \setminus N_G(y_{j_1})) \cap B_i(a)$ and avoid $y_{j_2}$; let $Y := (Y \cup \{y_{\alpha}\}) \setminus \{y_{j_2}\}$. Case (I-iii) $d_G(a_w, y_{j_1}) = 1$. Now, we can avoid $y_{j_2}$; let $Y := Y \setminus \{y_{j_2}\}$. Case (II) $y_{j_1} = x_k$. By Fact \ref{fact}, it follows that $(N_G(a_w) \setminus N_G(y_{j_1})) \cap B_i(a) \ne \emptyset$. Now, we can take $y_{\alpha} \in (N_G(a_w) \setminus N_G(y_{j_1})) \cap B_i(a)$ and avoid $y_{j_2}$; let $Y := (Y \cup \{y_{\alpha}\}) \setminus \{y_{j_2}\}$. Case (II-a) While, for some $a_{-1} \in N_G(y_{j_1}) \setminus \{y_{j_2}\}$, we can also consider a path $a_{-1}y_{j_1}y_{j_2}P_{1}y_{1}$. Case (II-a-i) $d_G(y_{j_2}, a_{-1}) = 2$. By Fact \ref{fact}, it follows that $(N_G(a_{-1}) \setminus N_G(y_{j_2})) \cap B_i(a) \ne \emptyset$. Now, we can take $y_{\alpha} \in (N_G(a_{-1}) \setminus N_G(y_{j_2})) \cap B_i(a)$ and avoid $y_{j_1}$; let $Y := (Y \cup \{y_{\alpha}\}) \setminus \{y_{j_1}\}$. Case (II-a-ii) $d_G(y_{j_2}, a_{-1}) = 1$. Now, we can avoid $y_{j_1}$; let $Y := Y \setminus \{y_{j_1}\}$. That is, $y_k \in N_G[x_k] \cap Y$ for $1 \leq k \leq q$ and $p = q$ are possible.
\end{proof}

By Fact \ref{fact3}, $Y$ is a dominating set of $G$. Since $G$ is regular and $p = q$, $Y$ is an efficient dominating set of $G$, where $Y \subseteq B_i(a)$. It contradicts that for all $h \in \{1, 2, \cdots, l\}$, $X_h \not \subseteq B_i(a)$. The proof is complete. 

\end{proof}

Let $A_0 := V(G)$ and $i := 0$. We follow the next steps. 
\begin{enumerate}
\item[(1)] Drop a vertex $v \in A_i$ that satisfies Fact \ref{fact} from $A_i$, and let  $A_{i + 1} := A_i \setminus \{v\}$.
\item[(2)] Increment $i$, and go to (1).
\end{enumerate}
By these steps, we have the final set $A_j$ ($j \geq 0$). \\

Let $i := j$. We follow the next steps. 
\begin{enumerate}
\item[(3)] Select a vertex $a \in A_i$.
\item[(4)] Let $l_1 := i$.
\item[(5)] Select a vertex $a' \in N_G[a] \cap A_i$, and drop $N_G(a') \cap A_i$ and $N_G^{2}(a') \cap A_i$ from $A_i$, and let $A_{i + 1} := A_i \setminus (N_G(a') \cup N_G^{2}(a'))$.
\item[(6)] Drop a vertex $v \in A_{i + 1}$ that satisfies Fact \ref{fact} from $A_{i + 1}$, and let $A_{i + 2} := A_{i + 1} \setminus \{v\}$.
\item[(7)] Increment $i$, and go to (6).
\item[(8)] Let $l_2 := i + 2$. If the final set $A_{l_2} = \emptyset$, then let $i := l_1$ and go to (5). 
\end{enumerate}
By these steps, we have the final set $A_{k_1}$ ($k_1 \geq j$). Let $i := k_1$ and follow the steps (3)-(8). By these steps, we have the final set $A_{k_2}$ ($k_2 \geq k_1$). By repeating this, we have the final set $A_{k_m}$ ($k_m \geq \cdots \geq k_2 \geq k_1$, $m \geq 1$).

\begin{prop}\label{cc}
$\mathcal{X}_0 = \emptyset$ if and only if $A_{k_m} = \emptyset$. Here, $A_{k_m}$ is determined in polynomial time.
\end{prop}

\begin{proof}
By step (1) - (2), we drop the vertices that satisfy Fact \ref{fact}, that is, the vertices that are not contained in all $X \in \mathcal{X}_0$. If $A_j = \emptyset$, then $\mathcal{X}_0 = \emptyset$. By step (3) - (7), if $A_{l_2} = \emptyset$, then by Proposition \ref{fact4}, $a' \not \in X$ for all $X \in \mathcal{X}_0$. If $\mathcal{X}_0 \ne \emptyset$, then for some $a'$, $A_{l_2} \ne \emptyset$, and so $A_{l_2}$ is a dominating set. By repeating step (3) - (8), we have the final set $A_{k_m} = X$ for some $X \in \mathcal{X}_0$. Obviously, if $\mathcal{X}_0 = \emptyset$ then $A_{k_m} = \emptyset$. All the steps are in polynomial time.
\end{proof}

By Proposition \ref{cc}, we complete the proof of Theorem \ref{thm1}.

\end{proof}


\begin{thebibliography}{}
%
%


\bibitem{Diestel}
Reinhard Diestel: Graph Theory Fourth Edition. Springer (2010)








\end{thebibliography}
\end{document}